\documentclass[a4paper]{amsart}
\usepackage{amsfonts,amsmath,amsthm,amssymb,graphicx,appendix,enumitem,color,hyperref}
\usepackage{dsfont}

\newtheorem{theorem}{Theorem}
\newtheorem{lemma}{Lemma}

\title{Exponential Bounds and Analyticity for the Tree Builder Random Walk}
\author{Caio Alves}
\address[Caio Alves]{FH Technikum Wien, Vienna, Austria}
\email[Caio Alves]{caio\_teodoro.de\_magalhaes\_alves@technikum-wien.at}
\author{Rodrigo Ribeiro}
\address[Rodrigo Ribeiro]{IMPA Tech, Rio de Janeiro, Brazil}
\email[Rodrigo Ribeiro]{rodrigo.ribeiro@impatech.edu.br}

\thanks{R.\ Ribeiro is affiliated with IMPA Tech, Rio de Janeiro, Brazil. E-mail: rodrigo.ribeiro@impatech.edu.br \\ C.\ Alves is affiliated with FH Technikum Wien, Vienna, Austria. Email: caio\_teodoro.de\_magalhaes\_alves@technikum-wien.at}

\newcommand{\Law}{\mathcal{L}}
\newcommand{\hitfetak}{\widetilde{H}}
\newcommand{\Lbrace}{\left\lbrace}
\newcommand{\Rbrace}{\right\rbrace}
\newcommand{\D}[1]{{\rm dist}(X_{#1},o)}
\newcommand{\Deg}[1]{{\rm deg}(X_{#1})}
\newcommand{\mutau}{\mu_{\tau}}
\newcommand{\mudtau}{\mu_{D}}

\begin{document}

\begin{abstract}
    In this work we investigate a class of random walks that interacts with its environment called Tree Builder Random Walk (TBRW). In our settings, at each step, the walker adds a random number of vertices to its position sampled according to a distribution $Q$. Previous works showed that the walker is ballistic with a well-defined speed, and that the TBRW admits a renewal structure, meaning that the process can be split into i.i.d epochs. We show that the first renewal time has exponential tail. Moreover, we show two consequences of the light tail of the first renewal time:  an exponential upper bound for the empirical speed of the walker, and, for the case in which the walker adds only one vertex with probability $p$, we show that the limiting speed is an analytic function of the parameter $p$. In some of our proofs, we apply techniques from bond percolation, which consist of extending probabilities to the complex numbers and using the Weierstrass $M$-test. 
\end{abstract}

\maketitle

\section{Introduction}

The Tree Builder Random Walk (TBRW) is a class of random walks that interacts with its environment by modifying the underlying graph. At each step, the walker can add a random number of vertices to its current position, and then takes a simple random walk step on the (possibly) updated graph. This feedback mechanism, in which the walker actively shapes the geometry of its environment, sets the TBRW apart from classical Random Walk on Random Environment (RWRE) models \cite{solomon1975,zeitouni2004}, where the environment is fixed, and from reinforced random walk models \cite{pemantle2007}, where the walker modifies transition probabilities but not the graph itself.

More broadly, the interaction between random walks and random media remains a highly active area. Recent progress includes fundamental advances in the random conductance model \cite{biskup2011}, connections between reinforced random walks and supersymmetric spin models \cite{sabot2015}, the analysis of once-reinforced random walks on trees \cite{collevecchio2020}, and law of large numbers for random walks in dynamic random environments \cite{avena2009}. 

The TBRW was introduced in \cite{figueiredo2017building}, where the authors showed that, in the Bernoulli case, where a new vertex is added with probability $p$ at each step, the walker is transient and ballistic with a well-defined speed $v(p)$, that is,
$$
\lim_{n\to \infty}\frac{{\text{dist}}(X_n,o)}{n} = v(p) > 0,
$$
almost surely. The general case, where at each step the number of added vertices follows an arbitrary distribution $Q$, was addressed in \cite{IRVZ22}, where the authors established transience and ballisticity under a uniformly elliptic condition on $Q$. In \cite{ribeiro2023}, this was extended to a full Strong Law of Large Numbers with a well-defined speed $v(Q)$, a Central Limit Theorem, and a Law of the Iterated Logarithm, via the construction of a renewal structure for the TBRW. Whereas in \cite{blanc2026}, the authors investigated a biased version of the TBRW, showing that the walker can be either ballistic or recurrent depending on the bias parameter.

The key tool behind these limit theorems is the renewal structure itself: the walker's trajectory decomposes into i.i.d.\ epochs, which allows one to transfer classical results for i.i.d.\ sequences to the study of the walker. This approach has a long and fruitful history in random walk theory; 
notably, Sznitman and Zerner \cite{sznitman_zerner1999} used renewal structures to prove ballisticity and a SLLN for RWRE in $\mathbb{Z}^d$ under condition~$(T')$, and Sznitman \cite{sznitman2000} later obtained a CLT under moment conditions on the first regeneration time.

A recurring theme in this program is that the finer properties of the walker, CLT, large deviations, regularity of the speed, are all governed by the tail behavior of the first renewal time $\tau_1$. In \cite{ribeiro2023}, the authors showed a stretched exponential bound $P_Q(\tau_1 \ge n) \le e^{-cn^{1/8}}$, which already suffices for the CLT and the LIL. However, sharper tails unlock stronger results.

In this paper, we substantially improve the tail bound on $\tau_1$ and exploit it to derive two new results for the TBRW:
\begin{enumerate}[label=(\roman*)]
    \item {\bf Exponential concentration of the speed.} We prove that the probability that the empirical speed ${\rm dist}(X_n,o)/n$ exceeds the limiting speed $v(Q)$ by some $\varepsilon$ is exponentially small in $n$, uniformly over the family of uniformly elliptic distributions (Theorem~\ref{t:ld}).
    \item {\bf Analyticity of the speed.} In the Bernoulli case $Q = {\rm Ber}(p)$, we show that $v(p)$ is an analytic function of $p \in (0,1]$ (Theorem~\ref{t:v}). Our approach adapts techniques from bond percolation developed in \cite{grego}, extending probabilities to the complex plane and applying the Weierstrass $M$-test, with the exponential tail of $\tau_1$ providing the necessary summability.
\end{enumerate}
In order to state our results precisely, let us introduce the model and the necessary notation.

\subsection{The Tree Builder Random Walk (TBRW)}\label{s:bgrw}
The TBRW generates a sequence of pairs $(T_n,X_n)$, where $T_n$ denotes the rooted tree at time $n$ and $X_n$ is a vertex of $T_n$ which represents the position of the walker at time $n$. The model is defined in a markovian fashion as follows: fix an initial state $(T_0,X_0)$, where $T_0$ is a rooted finite tree of root $o$ and $X_0$ is a vertex of $T_0$, and a sequence of probability distribution over the natural numbers $\{\Law_n\}_n$, then for each $n$
    \begin{itemize}
        \item obtain $T_{n+1}$ from $(T_n,X_n)$ by adding random number of vertices to $X_n$. This random number of vertices is sampled according to $\Law_{n+1}$;
        \item obtain $X_{n+1}$ by moving $X_n$ to a uniformly selected neighbor of it in $T_{n+1}$;
    \end{itemize}
In words, we first add a random number of vertices to $X_n$ according to the distribution $\Law_{n+1}$, and then let $X_n$ take a single step of a simple random walk on the, possibly updated, tree $T_{n+1}$. 

To illustrate the model and how the parameter $p$, in the Bernoulli case $Q= {\rm Ber}(p)$, shapes the tree, Figure~\ref{fig:tbrw-trees} shows three independent realizations of the TBRW for $p = 0.1, p=0.5$ and $p=0.9$. In each case, the number of steps is chosen as $n = \lfloor 99/p \rfloor$ so that the expected number of vertices is approximately $100$, for a fair comparison across different regimes.

\begin{figure}[ht]
  \centering
  \includegraphics[width=0.32\textwidth]{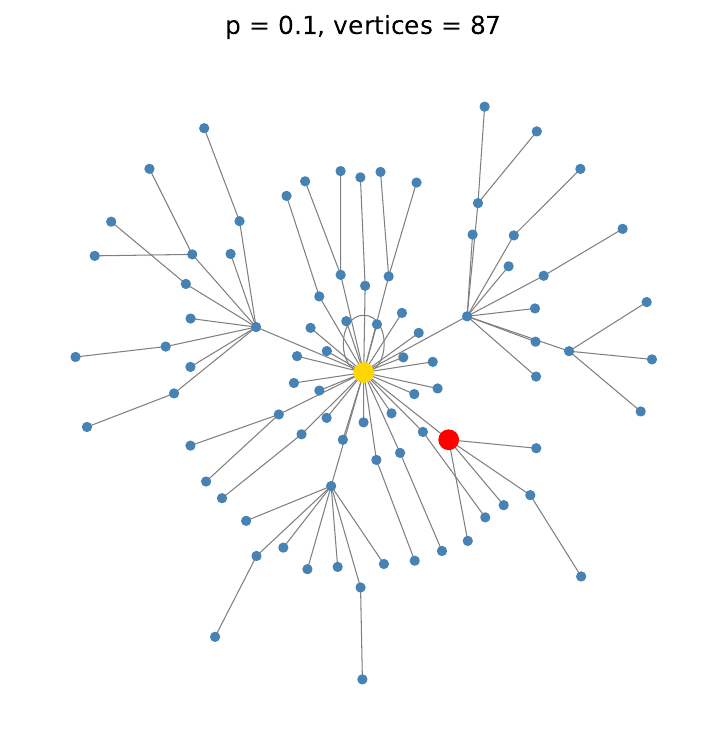}\hfill
  \includegraphics[width=0.32\textwidth]{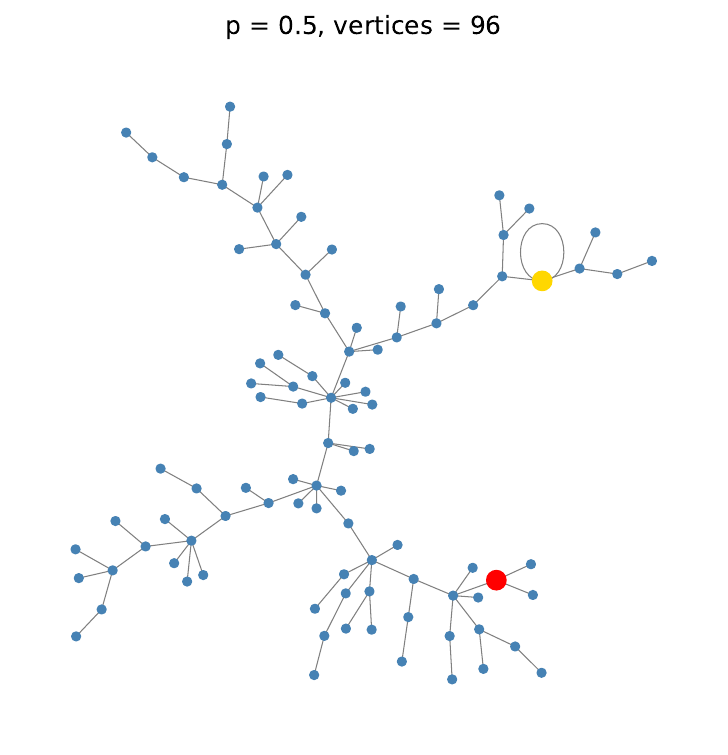}\hfill
  \includegraphics[width=0.32\textwidth]{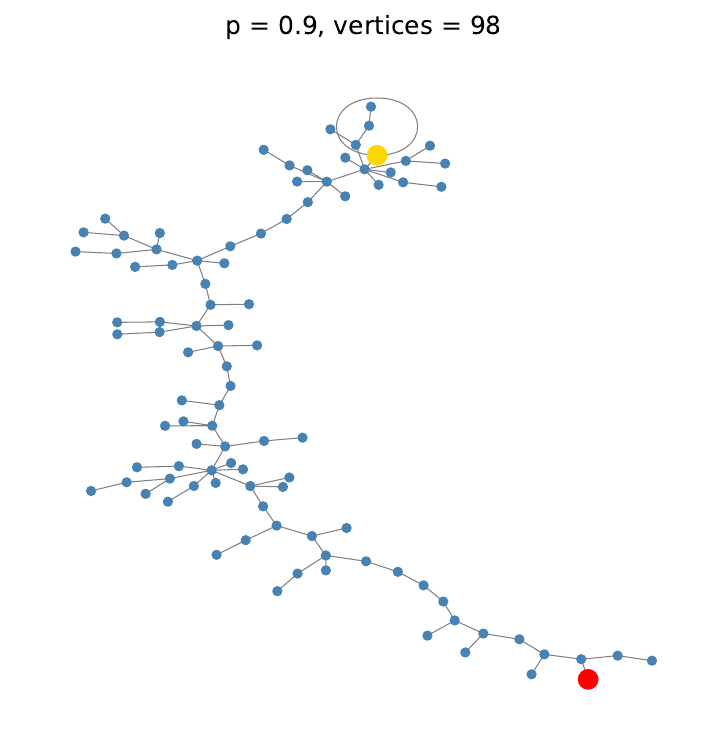}
  \caption{TBRW trees with ${\sim}100$ expected vertices for $p = 0.1$, $p = 0.5$, and $p = 0.9$.}
  \label{fig:tbrw-trees}
\end{figure}

\subsubsection{Basic Notation and Conventions}
We will make use of the notation $\mathcal{F}_n$ for the natural filtration, that is,
\begin{equation}\label{eq:Fn}
    \mathcal{F}_n := \sigma(T_0,X_0,T_1,X_1,T_2,X_2,\dots,T_n,X_n).
\end{equation}
Notice that there is no need to add information about value of the number of vertices added at each step, since they are obtained by keeping track on the changes on the trees from one step to the other. We will reserve the notation $\mathcal{F}$ to the smallest sigma algebra that makes the whole process measurable. And the usual notation $(\Omega, \mathcal{F})$ for the measurable space.

In this paper we will investigate the instance of the model when the sequence of laws $\{\Law_n\}_n$ is set to be $\Law_n = Q$ for all $n$. In this case, we will write $P_{T_0,X_0; Q}$ (and $E_{T_0,X_0;Q}$ for the associated expectation), for the distribution of the TBRW with initial state $(T_0,X_0)$ and $\Law_n = Q$ for all $n$. We refer to the distribution $Q$ as the {\it leaf distribution} since it controls the number of leaves added at each step. When the initial state is an edge with $X_0$ at the non-root tip of it, we omit the initial state for simplicity, writing $P_Q$ (respec. $E_Q$). 

We will leverage the tree structure of $T_n$ quite often, and it will be useful to apply genealogy terms to refer to vertices. In this direction, given a vertex $x$ on a tree $T$, we will write $f(x)$ for its {\it father}, that is, for the neighbor of $x$ on $T$ that is closer to the root $o$. Also, we will denote by $\D{n}$ the (graph) distance of the walker at time $n$ from the root $o$.

We will write $|T|$ for the number of vertices on $T$ and $h(T)$ for its height. And finally, $\Deg{n}$ denotes the degree of $X_n$ in $T_n$, that is, the number of neighbors of $X_n$ at time $n$.

\subsection{The Uniformly Elliptic Condition}
In the TBRW, the notion of ballisticity is connected to a condition on the sequence of laws $\Law$. We say a sequence of probability distributions $\Law = \{\Law_n\}_n$ over the natural numbers is said to be {\it uniformly elliptic} if there exists $\kappa \in (0,1]$ such that
\begin{equation}\label{eq:ue}
    \inf_{n \in \mathbb{N}}\Law_n(\{1,2,\dots\}) = \kappa.
\end{equation}
A TBRW with a uniformly elliptic sequence $\Law$ has the feature that at each step the walker adds at least one leaf to its position with probability at least $\kappa$. In \cite{IRVZ22}, the authors proved that under this settings, the TBRW is ballistic.

For a given $\kappa \in (0,1]$, we let $\mathcal{Q}_\kappa$ be the family of probability distribution over $\mathbb{N}$ such that $Q(\{1,2,\dots\})\ge \kappa$.

\subsection{Main Results}
In \cite{ribeiro2023}, the authors constructed a renewal structure for the TBRW where the first renewal time $\tau_1$ is defined as follows.
\begin{equation}\label{def:tau1}
    \begin{split}
        \tau_1 =
        \inf \Lbrace
        \begin{array}{c}
            n >0 :\Deg{n} = 1, \D{s} < \D{n} \le \D{t}, \\ 
            \forall s <n, \forall t>n
        \end{array}
        \Rbrace.
    \end{split}
\end{equation}
In words, $\tau_1$ is the first time the walker reaches a certain distance $\D{n}$ by jumping to a leaf and then it never visits the father of that leaf.

Our first main result shows that $\tau_1$ has exponentially light tail. Moreover, the rate can be chosen uniformly across all distributions $Q \in \mathcal{Q}_\kappa$ for some $\kappa \in (0,1]$. More precisely, we have the following.
\begin{theorem}[Exponential tail of $\tau_1$]\label{t:tautail} Consider a TBRW with leaf distribution $Q$ and fix $\kappa\in (0,1]$, then there exists a constant $c$ depending on $\kappa$ only, such that
$$
\sup_{Q \in \mathcal{Q}_\kappa}P_Q(\tau_1 \ge n) \le e^{-cn}.
$$
    
\end{theorem}
In \cite{ribeiro2023}, the authors showed a stretched exponential tail bound for $\tau_1$, that is, their Theorem 1 guarantees that
\begin{equation}\label{eq:stretched}
\sup_{Q\in \mathcal{Q}_\kappa}P_Q(\tau_1 \ge n) \le e^{-cn^{1/8}},    
\end{equation}
for some constant $c$ depending on $\kappa$ only. The way the above bound is obtained involves keeping track on the times the walker visits leaves at a further level and the times it returns to the father of those leaves. Then the event $\{\tau_1 > n\}$ is rewritten as a disjoint union of events involving how many times the walker visited leaves before time $n$. An union bound is then applied to obtain the above bound.

Our approach to prove Theorem \ref{t:tautail} overcomes the union bound done in \cite{ribeiro2023}. This new faster decay is possible due to a new decomposition of $\tau_1$ which we describe at Section \ref{s:prooftautail}. The key aspects of our decomposition are: whenever the walker does an excursion and returns to the father of a vertex, these excursions cannot be too long, otherwise, it is too expensive for the walker to make its way to the father of a vertex; thus, the walker is not too far way from a leaf, which means that it does not have to wait too long in order to have a new chance to regenerate from a leaf again. The upshot is that it is then possible to write $\tau_1$ as a random sum of random variables with exponentially light tails, where this random index obeys a geometric distribution. 

In \cite{figueiredo2017building} the authors showed that TBRW with $Q = {\rm Ber}(p)$ for $p \in (0,1]$ has a well-defined speed, that is, there exists a positive constant $v = v(p)$ such that
$$
\lim_{n\to \infty }\frac{\D{n}}{n} = v(p),
$$
almost surely. In \cite{ribeiro2023}, the authors generalize the above result and show that a TBRW with a leaf distribution $Q \in \mathcal{Q}_\kappa$ for some positive $\kappa$ has a well-defined speed as well, that is,
\begin{equation}\label{eq:vQ}
\lim_{n\to \infty }\frac{\D{n}}{n} = v(Q),
\end{equation}
almost surely, for some constant $v$ that depends only on $Q$.

In our second main result, we leverage the bound given by our Theorem \ref{t:tautail} and show an exponential concentration result for $\D{n}$ which also works uniformly on $\mathcal{Q}_\kappa$. That is, we show that the probability that $\D{n}$ deviates from $v(Q)n$ by $\varepsilon n$ decays exponentially fast in $n$. More specifically, we show the following result.
\begin{theorem}[Exponential concentration]\label{t:ld} Consider a TBRW with leaf distribution $Q \in \mathcal{Q}_{\kappa}$ for a fixed $\kappa \in (0,1]$. Then, there exists a positive constant $c=c(\varepsilon)$ such that
$$
\sup_{Q \in \mathcal{Q}_\kappa}P_Q\left(\left| \frac{\D{n}}{n} - v(Q)\right|>\varepsilon \right) \le e^{-cn}.
$$    
\end{theorem}
Our final main result is about the regularity of the speed $v(p)$ as a function of $p$ in the Bernoulli case. In \cite{ribeiro2023}, the authors showed that $v(p)$ is continuous on $p$. Here, we substantially improve their result by showing that $v(p)$ is actually an analytic function of $p$. More precisely, we have the following result.
\begin{theorem}[The speed is analytic]\label{t:v}
Consider a TBRW with $Q = {\rm Ber}(p)$, for $p \in (0,1]$. Then, the speed $v(p)$ is analytic on $p$.    
\end{theorem}
The rest of this paper is organized as follows: at Section \ref{s:prooftautail} we introduce the renewal structure formally and show Theorem \ref{t:tautail}. At Section \ref{s:ld}, we show Theorem \ref{t:ld}. At Section \ref{s:analyticity}, we show Theorem \ref{t:v}. We then conclude in Section \ref{s:open} with a discussion of open problems arising from this work.

\section{Proof of Theorem \ref{t:tautail}: Exponential tail bounds for $\tau_1$ }\label{s:prooftautail}
 In this part, we will define the renewal structure and show that the first renewal time has exponential tail. 

Before we define the renewal structure, let us say a couple words about the markovian nature of the TBRW. Observe that the TBRW itself is a Markov process on the space of rooted trees, even though each of its components separately are not. In some proofs, it will be useful to make use of the (Strong) Markov Properties. For that, we will denote the shift operator of length $n$ by $\theta_n$. In light of the Markov Properties, it is important to point out that when the TBRW is shifted by $n$ steps and $\Law_n = Q$ for all $n$, the shifted process is a TBRW that starts from a random initial state $(T_n,X_n)$ and with the exact same sequence of laws $\{\Law_n\}_n$. More formally, if $\phi: \Omega \to \mathbb{R}$ is $\mathcal{F}$-measurable, then the Simple Markov Property works as follows
$$
E_{T_0,X_0;Q}\left [\phi\circ \theta_n \; \middle | \; \mathcal{F}_n\right] = E_{T_n,X_n;Q}\left [\phi\right].
$$

\subsection{The Renewal Structure}
Recall the definition of $\tau_1$ in \eqref{def:tau1}. The renewal structure is then defined inductively
\begin{equation}\label{eq:tauk}
\tau_k := \tau_1 \circ \theta_{\tau_{k-1}} + \tau_{k-1}.
\end{equation}
Given a vertex $x$, we will write $H_x$ for the hitting time to $x$, that is, the first time $n$ when $X_n = x$. We will keep track on the times the walker jumps to leaves at maximal distance from the root and also when it backtracks to the father of these special leaves. In order to do so, we will introduce two sequences of stopping times whose definition depend on each other. We start defining 
$$
\zeta_0 \equiv 0; \quad \hitfetak_0 = H_{f(X_0)}.
$$
Then, we let $\zeta_1$ be the following stopping time
$$
\zeta_1 = \inf \{ n> \hitfetak_0 \; : \;\Deg{n} = 1, \D{n} = h(T_n), \D{n}>\D{0} \}.
$$
In words, $\zeta_1$ happens when the following conditions take place: after the walker visits the father of $X_0$, it visits a leaf with maximal distance from the root, this leaf also having a greater distance from the root than the previous leaf, $X_0$. Note that never visiting the father of $X_0$ again would imply that $\tau_1$ happens at time $0$.

We then put $\hitfetak_1$ to be
\begin{equation}
    \hitfetak_1 := \inf\Lbrace n > \zeta_1 \; : \; X_n = f\left( X_{\zeta_1} \right)\Rbrace,
\end{equation}
with the convention that the infimum over the empty set is infinity. In words, $\hitfetak_1$ is the first time the walker returns to the father of $X_{\zeta_1} $ after stepping on $X_{\zeta_1} $ for the first time. When the walker is transient, $\hitfetak_1$ can be infinity.

Then, we define $\zeta_k$ and $\hitfetak_k$ for $k>1$ inductively as follows
\begin{equation}
    \label{eq:etakdef}
    \zeta_k = \begin{cases}
                \zeta_1 \circ \theta_{\hitfetak_{k-1}} + \hitfetak_{k-1}, & \text{ if } \hitfetak_{k-1} <\infty  \\
                \infty, & \hitfetak_{k-1} = \infty 
            \end{cases},
\end{equation}
again with the convention that $\inf\emptyset = \infty$.
\begin{equation}
    \label{eq:hitfkdef}
    \hitfetak_k = H_{f(X_0)} \circ \theta_{\zeta_{k}} + \zeta_{k}. 
\end{equation}
In words, $\zeta_k$ happens when, after visiting the father of $X_{\zeta_{k-1}}$, the walker visits a leaf at maximal distance from the root, and even further away from the root then the previous leaf visit at $\zeta_{k-1}$. Furthermore, $\hitfetak_k$ is the first time the walker returns to the father of $X_{\zeta_{k}}$.

Recall that $\tau_1$ is defined only at times for which a unique new maximum for the distance between the walker and the root is achieved. The times $(\zeta_k)_{k \geq 1}$ are the ones in which a regeneration could happen, but if $\tilde{H}_k$ is finite, then the regeneration fails. With that in mind, we define $K$ to be the random index
\begin{equation}\label{eq:Kdef}
    K = \inf\{ k \ge 0 \; : \; \hitfetak_k = \infty\},
\end{equation}
and notice that we can define $\tau_1$ as
\begin{equation}
    \label{eq:tauetakdef}
    \tau_1 := \zeta_{K} = \sum_{j=1}^K\zeta_j -\zeta_{j-1}.
\end{equation}
The fact that we can write $\tau_1$ as the above sum will be crucial to us. This new standpoint is already telling us that the tail of $\tau_1$ should be as light as the tails of $\zeta_j - \zeta_{j-1}$ and $K$ provided they behave well enough. The next two lemmas go exactly in this direction dealing with the distributions of $K$ and $\zeta_j - \zeta_{j-1}$.
\begin{lemma}\label{l:Kdist}Let $K$ be as in \eqref{eq:Kdef}. Then, for all $Q \neq \delta_0$, $K$ obeys a geometric distribution of parameter $P_Q(H_o = \infty)$.
    
\end{lemma}
To avoid clutter with the notation, we will make use of the $\Delta$ notation for the difference between two random variables in a sequence. That is, for all $k \in \mathbb{N}$, we write
\begin{equation}\label{eq:delta}
    \Delta \zeta_k := \zeta_k - \zeta_{k-1}.
\end{equation}
\begin{lemma}\label{l:exp} Fix $\kappa \in (0,1]$. Then there exists $t_0 = t_0(\kappa)>0$ such that for all $t<t_0$ and $j\ge 1$ 
    $$
    \sup_{Q \in \mathcal{Q}_\kappa}E_Q \left[\exp\Lbrace t\Delta \zeta_j \mathds{1}\{\hitfetak_{j-1} < \infty \} \Rbrace \; \middle | \mathcal{F}_{\zeta_{j-1}}\; \right] \le e^{t}\left[P_\kappa(H_o < \infty) + \epsilon(t)\right],
    $$
    where $\epsilon(t)$ is a function of $t$ depending on $\kappa$ only, satisfying $\epsilon(t)\to 0$, when $t\to 0$.

\end{lemma}
We will postpone the proof of the above lemmas to the end of this section. For now, we will focus on showing how Theorem \ref{t:tautail} follows from them.
\begin{proof}[Proof of Theorem \ref{t:tautail} (Exponential Tail Bounds for $\tau_1$)] 
Since $\tau_1 = \zeta_K$ and the events $\{K = k\}$ partition $\Omega$, since the walker is transient in our settings, we can decompose the probability of $\{\tau_1 \ge m\}$ as
\begin{equation}\label{eq:tausum}
    P_Q( \tau_1 \ge m) = \sum_{k=1}^\infty P_Q\left( \sum_{j=1}^{k}\Delta \zeta_j \ge  m, \; K = k\right).
\end{equation}
On the event $\{K = k\}$, we have $\hitfetak_0 < \infty, \dots, \hitfetak_{k-1} < \infty$ and $\hitfetak_k = \infty$, so in particular $\zeta_k < \infty$ (since in these settings the walker is transient) and the walker is at a leaf at time~$\zeta_k$, by definition of $\zeta_k$. Using these observations and shifting the process by $\zeta_k$, Strong Markov property yields
\begin{equation}\label{eq:smpho}
    P_Q\left( \hitfetak_{k} = \infty \; \middle | \; \mathcal{F}_{\zeta_k} \right)\mathds{1}\{\zeta_k <\infty\} = P_Q(H_o = \infty)\mathds{1}\{\zeta_k <\infty\}, \; P_Q\text{-a.s.}
\end{equation}
Applying the above identity at $\zeta_k$ on the event $\{K = k\}$ (where $\zeta_k < \infty$) and then dropping the conditions $\hitfetak_0 < \infty, \dots, \hitfetak_{k-1} < \infty$ gives the upper bound
\begin{equation}\label{eq:bound1}
    P_Q\left( \sum_{j=1}^{k}\Delta \zeta_j \ge  m, \; K = k\right) \le P_Q(H_o = \infty)\,P_Q\left( \sum_{j=1}^{k}\Delta \zeta_j \mathds{1}\{\hitfetak_{j-1} < \infty \}\ge  m\right).
\end{equation}
Notice that for any $t>0$, an application of Markov's inequality together with Lemma \ref{l:exp} $k$ times yields 
\begin{equation}
    \begin{split}
        P_Q\left( \sum_{j=1}^{k}\Delta \zeta_j \mathds{1}\{\hitfetak_{j-1} < \infty \}\ge  m\right) & = P_Q\left( \exp\Lbrace\sum_{j=1}^{k}t\Delta \zeta_j \mathds{1}\{\hitfetak_{j-1} < \infty \} \Rbrace \ge  e^{tm}\right) \\
        & \le e^{-tm}E_Q \left[\exp\Lbrace\sum_{j=1}^{k}t\Delta \zeta_j \mathds{1}\{\hitfetak_{j-1} < \infty \} \Rbrace\right] \\
        & \le e^{-tm}e^{tk}\left[P_\kappa(H_o < \infty) + \epsilon(t)\right]^k.
    \end{split}
\end{equation}
Choosing $t$ small enough so that $e^{t}\left[P_\kappa(H_o < \infty) + \epsilon(t)\right] < 1$, we obtain 
\begin{equation}
    \begin{split}
    \sum_{k=1}^\infty P_Q\left( \sum_{j=1}^{k}\Delta \zeta_j \mathds{1}\{\hitfetak_{j-1} < \infty \}\ge  m\right) & \le e^{-tm}\sum_{k=1}^\infty e^{tk}\left[P_\kappa(H_o < \infty) + \epsilon(t)\right]^k\\
    & \le \frac{e^{-tm}e^{t}\left[P_\kappa(H_o < \infty) + \epsilon(t)\right]}{1-e^{t}\left[P_\kappa(H_o < \infty) + \epsilon(t)\right]}.
    \end{split}
\end{equation}
Plugging the above inequality back into \eqref{eq:bound1} and using \eqref{eq:tausum} yields
$$
P_Q(\tau_1 \ge m) \le \frac{e^{-tm}e^{t}P_\kappa(H_o = \infty)\left[P_\kappa(H_o < \infty) + \epsilon(t)\right]}{1-e^{t}\left[P_\kappa(H_o < \infty) + \epsilon(t)\right]},
$$
which is enough to prove the theorem for a suitable adjustment of the constants.
\end{proof}
\subsection{Proofs of Lemmas \ref{l:Kdist} and \ref{l:exp}}
We begin with proof of Lemma \ref{l:Kdist} which is a consequence of Strong Markov property.
\begin{proof}[Proof of Lemma \ref{l:Kdist}] For any fixed $k \in \mathbb{N}$, the definition of $K$ in \eqref{eq:Kdef} gives us that
\begin{equation}
    P_Q(K=k) = P_Q(\hitfetak_1 <\infty, \dots, \hitfetak_{k-1} < \infty, \hitfetak_k = \infty).
\end{equation}
    Since the $\hitfetak$'s are stopping times, using \eqref{eq:smpho} leads us to
    $$
        P_Q(\hitfetak_1 <\infty, \dots, \hitfetak_{k-1} < \infty, \hitfetak_k = \infty) = P_Q(\hitfetak_1 <\infty, \dots, \hitfetak_{k-1} < \infty)P_Q( H_o = \infty).
    $$
    To finish the proof, recall that at time $\zeta_{k-1}$, the walker is at a leaf of maximal distance from the root in $T_{\zeta_{k-1}}$. Also that $\hitfetak_i < \hitfetak_{i+1}$ on the event where $\hitfetak_i$ is finite. Thus, using Strong Markov Property again by shifting the process by $\zeta_{k-1}$
    $$
        P_Q(\hitfetak_1 <\infty, \dots, \hitfetak_{k-1} < \infty) = P_Q(\hitfetak_1 <\infty, \dots, \hitfetak_{k-2} < \infty)P_Q( H_o < \infty).
    $$
    Combining the two above equations and using induction finishes the proof.
\end{proof}
 Lemma \ref{l:exp} will follow from an exponential tail bound for $\Delta \zeta_j$ on the event that the walker returns to the father of $\zeta_{j-1}$. We will first state the result we need and show how Lemma \ref{l:exp} follows from it. Then we will finish this section showing the lemma.
 \begin{lemma}\label{l:Deltaz}Fix $\kappa \in (0,1]$ and $j\ge 1$. Then there exists $c = c(\kappa)>0$ such that
    $$
    \sup_{Q \in \mathcal{Q}_\kappa}\sup_{(T,x) \in \mathcal{T}_*}P_{T,x;Q}\left(\Delta \zeta_j > m, \hitfetak_{j-1} < \infty \; \middle | \; \mathcal{F}_{\zeta_{j-1}}\right ) \le 2e^{-cm}
    $$
     
\end{lemma}
We can now prove Lemma \ref{l:exp}
\begin{proof}[Proof of Lemma \ref{l:exp}] We will fix $t>0$ and choose it properly latter. Then we have that
\begin{equation}\label{eq:expsum}
    E_Q \left[\exp\Lbrace t\Delta \zeta_{j+1} \mathds{1}\{\hitfetak_{j} < \infty \} \Rbrace \; \middle | \mathcal{F}_{\zeta_{j}}\; \right] \le \sum_{m=1}^\infty P_Q\left( \Delta \zeta_{j+1} \ge \frac{\log(m)}{t}, \hitfetak_{j} < \infty\; \middle | \; \mathcal{F}_{\zeta_{j}}\right).
\end{equation}
    Notice that if $m$ is smaller than $e^t$, then by the Strong Markov Property
    \begin{equation}
        \begin{split}
            P_Q\left( \Delta \zeta_{j+1} \ge \frac{\log(m)}{t}, \hitfetak_{j} < \infty\; \middle | \; \mathcal{F}_{\zeta_{j}}\right) & = P_Q\left(  \hitfetak_{j} < \infty\; \middle | \; \mathcal{F}_{\zeta_{j}}\right)\\
            & = P_Q(H_o < \infty)\mathds{1}\{\zeta_j < \infty\},
        \end{split}
    \end{equation}
    where we have used that $\hitfetak_j$ is infinite if $\zeta_j$ is infinite.
On the other hand, if $m$ is greater than $e^t$, we can use Lemma \ref{l:Deltaz} to obtain
\begin{equation}
     P_Q\left( \Delta \zeta_{j+1} \ge \frac{\log(m)}{t}, \hitfetak_{j} < \infty\; \middle | \; \mathcal{F}_{\zeta_{j}}\right) \le 2e^{-c\log(m)/t}=2m^{-c/t},
\end{equation}
where $c$ depends only on the constant $\kappa$. Choosing $t<c$, using the above two identities on \eqref{eq:expsum}, and using \cite[Lemma 1]{ribeiro2023} which states that 
$$
P_Q(H_o < \infty) \le P_\kappa(H_o < \infty), \forall \in \mathcal{Q}_\kappa,
$$ 
we obtain the following bound
\begin{equation*}
    \begin{split}
        E_Q \left[\exp\Lbrace t\Delta \zeta_{j+1} \mathds{1}\{\hitfetak_{j} < \infty \} \Rbrace \; \middle | \mathcal{F}_{\zeta_{j}}\; \right] & \le e^tP_Q(H_o < \infty) + 2\sum_{m=1}^\infty m^{-c/t}\\
        & \le e^tP_\kappa(H_o < \infty) + 2t/(c-t) \\
        & = e^tP_\kappa(H_o < \infty)\left[1+\frac{2te^{-t}}{(c-t)P_\kappa(H_o < \infty)}\right].
    \end{split}
\end{equation*}
Finally, setting $\epsilon(t)$ to be
$$
\epsilon(t) := \frac{2t}{e^{t}(c-t)P_\kappa(H_o < \infty)},
$$

we conclude the proof.
\end{proof}
We are now left with showing Lemma \ref{l:Deltaz}. As usual, we need to introduce additional notation and a few auxiliary results. 

On the event where $\zeta_k$ is finite, we define $M_k$ to be as follows
\begin{equation}
    M_k = \sup_{\zeta_k \le n \le \hitfetak_k} \{\D{n} - \D{\zeta_k}\}.
\end{equation}
If $\hitfetak_k$ is infinite, then we set $M_k$ to be infinite as well. In words, on the event where~$\hitfetak_k$ is finite, $M_k$ returns how far $X$ went before returning to the father of $\zeta_k$.

In \cite{figueiredo2017building} the authors constructed a coupling of the distance process $\{\D{n}\}_n$ with a right-biased random walk $\{S_k\}_k$ on $\mathbb{Z}$ so that the right-biased random walk is always closer to $0$. More formally, in Lemma 5.12 of \cite{figueiredo2017building}, the authors showed that fixed $q \in (1/2,1)$, there exists a natural number $r$ that might depend on $Q$ and a sequence of stopping times $\{\sigma_k^{(r)}\}_k$ such that 
\begin{enumerate}
    \item  $|\sigma_k^{(r)} - \sigma_{k-1}^{(r)}| \le e^{\sqrt{r}}$, almost surely for all $k$
    \item  $|\D{\sigma^{(r)}_{k}} - \D{\sigma^{(r)}_{k-1}}| \le r$, almost surely for all $k$
    \item $\mathbb{P}\left(\D{\sigma^{(r)}_{k}} \ge rS_k\right) = 1$,
\end{enumerate}
where $\{S_k\}_k$ is a random walk on $\mathbb{Z}$ with probability $q$ of jumping to the right.
In Lemma 5 of \cite{ribeiro2023}, the authors showed that the constant $r$ can be chosen uniformly across all $Q\in \mathcal{Q}_k$, for $\kappa \in (0,1]$. This way, the constant $r$ depends only on the constant $\kappa$ and on $q$ and does not depend on the initial state $(T_0,x_0)$ or on other properties of the probability $Q$. Roughly speaking, this coupling says that all TBRW having a probability at least $\kappa$ of adding at least one vertex at each step are all being pushed away from the root at least as fast as the right-biased random walk $\{S_k\}_k$ is moving away from its initial position.

We will leverage this coupling to first show the lemma below which states that $M_k$ has exponential tails when the walker returns to the father of $X_{\zeta_k}$. Then, we will show Lemma \ref{l:Deltaz}.
\begin{lemma}\label{l:Mk} Fix $\kappa \in (0,1]$. Then there exists $c = c(\kappa)>0$ such that
    $$
    \sup_{Q \in \mathcal{Q}_\kappa}\sup_{(T,x) \in \mathcal{T}}P_{T,x;Q}(M_k > m, \hitfetak_k < \infty) \le e^{-cm}
    $$
\end{lemma}
\begin{proof}Let $H^*_{m}$ be the first time the walker reaches distance $m$ from its initial position $x$. That is,
\begin{equation}
    H^*_{m} := \inf \{ n > 0 \; :\; \D{n} - \D{0} > m\}.
\end{equation}
    Then, shifting the process by $\zeta_k$ and using Strong Markov Property, we have that 
    \begin{equation}\label{eq:smpmk}
        P_{T,x;Q}(M_k > m, \hitfetak_k <\infty) = E_{T,x;Q}\left[ \mathds{1}\{\zeta_k < \infty\}P_{T_{\zeta_k,}X_{\zeta_k};Q}\left( H^*_m < H_{f(X_0)} < \infty \right)\right],
    \end{equation}
    since the walker must reach distance $m$ from its initial position then visit the father of its initial position after that. Then, shifting the process by $H^*_m$ we have a process starting at distance $m$ from its initial position that manages to visit the father of its initial position in finite time. However, by the coupling with the right-biased random walk (fixing $q$ as, say, $2/3$) there exists a constant $c$ depending on $\kappa$ only such that 
    \begin{equation*}
        P_{T_{\zeta_k,}X_{\zeta_k};Q}\left( H^*_m < H_{f(X_0)} < \infty \right) \le e^{-cm}, \; P_{T,x;Q}\text{-a.s.}
    \end{equation*}
    Then, using the above bound in \eqref{eq:smpmk}, obtain
    $$
    P_{T,x;Q}(M_k > m, \hitfetak_k <\infty) \le e^{-cm}.
    $$
    which is what we desired.
\end{proof}
The reader may already have an idea on how Lemma \ref{l:Deltaz} follows from the above bound and the coupling with the right-biased random walk. The random increment~$\Delta \zeta_j$ measures the time it takes from the walker to take an excursion on the tree below $\zeta_{j-1}$ and climb back to the father of $\zeta_{j-1}$. On the other hand, by Lemma~\ref{l:Mk}, if the walker made that return, it did not go too far from $X_{\zeta_{j-1}}$. Thus, after stepping on the father of $\zeta_{j-1}$, the right-biased random walk will push the walker all the way down again at linear speed, making the walker cover the excursion it constructed before in linear time.

With the above in mind, we are finally ready to show Lemma \ref{l:Deltaz} which was the last result left to conclude the proof of Theorem\ref{t:tautail}.
\begin{proof}[Proof of Lemma \ref{l:Deltaz}] We would like to bound from above the following probability
\begin{equation*}
    \begin{split}
        P_{T,x;Q}\left(\Delta \zeta_j > m, \hitfetak_{j-1} < \infty \; \middle | \; \mathcal{F}_{\zeta_{j-1}}\right ) .
    \end{split}
\end{equation*}
    In order to do that, we will intersect the event $\{\Delta \zeta_j > m, \hitfetak_{j-1} < \infty\}$ with $\{M_{j-1} \le \varepsilon m \}$ and its complement. Notice that by Strong Markov Property and Lemma \ref{l:Mk}, it follows that 
    \begin{equation}\label{eq:expbound1}
         P_{T,x;Q}\left(M_{j-1} > \varepsilon m, \hitfetak_{j-1} < \infty \; \middle | \; \mathcal{F}_{\zeta_{j-1}}\right ) \le e^{-cm},\; P_{T,x;Q}\text{-a.s.},
    \end{equation}
    for some positive constant $c$ depending on $\varepsilon$ and $\kappa$. The next step is to bound
    \begin{equation*}
        P_{T,x;Q}\left(\Delta \zeta_j > m, M_{j-1} \le \varepsilon m, \hitfetak_{j-1} < \infty \; \middle | \; \mathcal{F}_{\zeta_{j-1}}\right ). 
    \end{equation*}
    Notice that by the construction of $\zeta_j$, when the process is shifted by $\zeta_{j-1}$, we have that the above probability equals
    \begin{equation*}
        P_{T_{\zeta_{j-1}},X_{\zeta_{j-1}};Q}(\zeta > m,M\le \varepsilon m, H_{f(X_0)} < \infty),
    \end{equation*}
    where $\zeta$ is the first time the walker reaches maximal distance after visiting the father of its initial position ($f(X_0)$) for the first time and $M$ is how far the walker traveled before visiting $f(X_0)$ for the first time. By shifting the process again by $H_{f(X_0)}$, we have that $\zeta$ measures the time to hit the bottom of $T_{H_{f(X_0)}}$, knowing that $f(X_0)$ is at distance at most $\varepsilon m$ from the bottom of $T_{H_{f(X_0)}}$. Using the coupling with the biased random walk, and make $\varepsilon$ sufficiently small, we know that the biased random walk cannot take too long to cover a distance of at most $\varepsilon m$. This discussion gives us that there exists another positive constant depending on $\varepsilon$ and $\kappa$ only, such that
    \begin{equation}\label{eq:expbound2}
        P_{T,x;Q}\left(\Delta \zeta_j > m, M_{j-1} \le \varepsilon m, \hitfetak_{j-1} < \infty \; \middle | \; \mathcal{F}_{\zeta_{j-1}}\right ) \le e^{-cm}, \; P\text{a.s.}
    \end{equation}
    Combining the above bound with \eqref{eq:expbound1} gives us that 
    \begin{equation*}
    \begin{split}
        P_{T,x;Q}\left(\Delta \zeta_j > m, \hitfetak_{j-1} < \infty \; \middle | \; \mathcal{F}_{\zeta_{j-1}}\right ) \le 2e^{-cm},
    \end{split}
\end{equation*}
for some constant $c$ depending on $\varepsilon$ and $\kappa$ only. This concludes the proof.
\end{proof}



\section{Proof of Theorem \ref{t:ld}: Exponential concentration}\label{s:ld}
Before showing Theorem \ref{t:ld}, we will need two auxiliary lemmas and introduce some additional notation to avoid clutter. Throughout this section, $c, c(\varepsilon)$ denote positive constants depending only on $\varepsilon$ and $\kappa$ whose value may change from line to line. We let $\mutau$ be
\begin{equation}\label{eq:mutau}
    \mutau := E_Q[ \tau_1 \mid H_o = \infty],
\end{equation}
and $\mudtau$ is defined as
\begin{equation}\label{eq:mudtau}
    \mudtau := E_Q[ \D{\tau_1}\mid H_o = \infty]
\end{equation}
It is also helpful to recall the definition of $\tau_m$ given at page \pageref{eq:tauk} Equation \eqref{eq:tauk}.
\begin{lemma}
\label{l:tau_ld_bound}
Fix $\kappa \in (0,1]$. Then, for all $\varepsilon > 0$, there exists $c = c(\varepsilon, \kappa)$ such that for all $m \in \mathbb{N}$ the following bound holds
    \begin{equation}
       \sup_{Q\in \mathcal{Q}_\kappa} P_Q \left(
                \left|
                    \tau_m - m \mutau
                \right|
                    >
                \varepsilon m
            \right)
            \leq
        e^{-c(\varepsilon) m}.
    \end{equation}  
\end{lemma}

\begin{proof} This lemma is a consequence of the renewal structure for the TBRW constructed in \cite{ribeiro2023} and our Theorem \ref{t:tautail}. By Theorem 2 of \cite{ribeiro2023}, we have that the random variables $\tau_1, \tau_2 - \tau_1, \dots, \tau_{k} - \tau_{k-1}, \dots$ are independent with $\tau_k - \tau_{k-1}$ distributed as $\tau_1$ conditioned on $\{H_o = \infty \}$.

Notice that for each $m$, we can write
$$
\tau_m =  \tau_1 + \sum_{k=2}^m (\tau_{k} - \tau_{k-1}),
$$
which is a sum of independent random variables with finite MGF by our Theorem~\ref{t:tautail}. The exponential decay then follows from Chernoff bounds. The fact that the bound holds uniformly over $Q \in \mathcal{Q}_\kappa$ also follows from Theorem \ref{t:tautail}, which holds uniformly over $Q \in \mathcal{Q}_\kappa$.
\end{proof}

For the second auxiliary lemma, we will need an extra definition. For each $t \in \mathbb{N}$, let $N(t)$ be as follows
\begin{equation}
    \label{eq:Ndef}
    N(t) = \sup\left\{\, k \, : \, \tau_k \leq t \right\}.
\end{equation}
That is, $N(t)$ counts the number of regenerations before time $t$. Our next result guarantees an exponential concentration for $N(t)$.
\begin{lemma}
    \label{l:N_ld_bound} Fix $\kappa \in (0,1]$. Then, for all $\varepsilon >0$, there exists a positive constant $c = c(\varepsilon,\kappa)$ such that
    \begin{equation}
       \sup_{Q\in \mathcal{Q}_\kappa} P_Q \left(
                \left|
                    N(t) - \frac{t}{\mutau}
                \right|
                    >
                \varepsilon t
            \right)
            \leq
        e^{- c t}.
    \end{equation}  
\end{lemma}
\begin{proof}
    We start by defining
    \begin{equation}
        \label{eq:mplusdef}
        m_+ = m_+(t, \varepsilon)
            = 
            \left(
                \frac{1}{\mutau} + 
                \varepsilon
            \right)
            t
    \end{equation}
    and therefore,
    \begin{equation}
        \label{eq:Ntldimplication}
        N(t) 
            >
        \left(
            \frac{1}{\mutau} + 
            \varepsilon
        \right)
        t
            \implies
                \tau_{\lceil m_+ \rceil} \leq t 
               = \mutau \lceil m_+ \rceil\left(1-\varepsilon\mutau + O(\varepsilon^2\mutau^2) \right),
    \end{equation}
    which is possible since, by Theorem \ref{t:tautail} and the definition of $\tau_1$, for all $Q \in \mathcal{Q}_\kappa$, $\mutau$ is bounded from above and from below by positive constants depending on $\kappa$ only. Finally, using Lemma \ref{l:tau_ld_bound}, we can make $\varepsilon$ small enough so that
    \begin{equation}
        \label{eq:Ntldbound2}
        P_Q \left(
            N(t) 
                >
            \left(
                \frac{1}{\mutau} + 
                \varepsilon
            \right)
            t
        \right)
        \leq
            \exp
            \left\{
                -c(\varepsilon) m_+
            \right\}
        =
            \exp
            \left\{
                -c(\varepsilon)  
                \left(
                    \frac{1}{\mutau} + 
                    \varepsilon
                \right)
                t
            \right\},
    \end{equation}
    since the events $\{N(t) > m_+\}$ and $\{\tau_{\lceil m_+ \rceil} \leq t\}$ are the same. The proof for the upper bound on $N(t)$ follows in the same manner.
\end{proof}

We can finally show Theorem \ref{t:ld} using the above two lemmas and Theorem \ref{t:tautail}.

\begin{proof}[Proof of Theorem \ref{t:ld}]

Note that $P_Q(N(t) = 0) = P_Q(\tau_1 > t) \leq e^{-ct}$ by Theorem~\ref{t:tautail}, so it suffices to work on the event $\{N(t) \geq 1\}$. In this event, $\D{t}$ can be written as
\begin{align}
    \D{t} \label{eq:ldproof1}
    & =
        \D{\tau_1}
        +
        \sum_{k = 1}^{N(t)-1}
            \left(
                \D{\tau_{k+1}}
                -
                \D{\tau_{k}}
            \right)
    \\ & \quad + \nonumber        
        \D{t} - \D{\tau_{N(t)}}.
\end{align}
We prove the bound for the probability of $\D{t}/t$ exceeding $v(Q)$ by $\varepsilon$, since the matching bound is proved in a similar way: on the event $\{\D{t} < (v(Q) - \varepsilon)t\}$, drop the nonnegative first and last summands on the right-hand side of \eqref{eq:ldproof1} to reduce to bounding the sum of i.i.d.\ increments from above, and then apply the same Chernoff argument. By the union bound, we can write
\begin{align}
    & P_Q 
        \left(
            \D{t} > 
            (\varepsilon + v(Q)) t
        \right)
    \nonumber
    \\
    &  \quad\quad\quad \label{eq:ldproof2}
    \leq    
    P_Q 
        \left(
            \D{\tau_1} > 
            \frac{\varepsilon}{3} t
        \right)
    \\
    & \quad\quad\quad \nonumber
    + 
    P_Q 
    \left(
        \sum_{k = 1}^{N(t)-1}
            \left(
                \D{\tau_{k+1}}
                -
                \D{\tau_{k}}
            \right) > 
        \left(
            \frac{\varepsilon}{3} 
            + 
            v(Q)
        \right)
        t
    \right)
    \\
    & \quad\quad\quad \nonumber
    +   
    P_Q 
        \left(
            \D{t} - \D{\tau_{N(t)}} > 
            \frac{\varepsilon}{3} t
        \right)
\end{align}
Since $\D{t} - \D{s} \leq t - s$ (the walker moves at most one step per unit time), we have
\[
    \D{\tau_1} 
        \leq \tau_1 + 1, 
    \quad 
    \D{t} - \D{\tau_{N(t)}} 
        \leq \tau_{N(t)+1} - \tau_{N(t)}.
\]
Recall that $\tau_{m+1} - \tau_{m}$ is distributed as $\tau_1$ conditioned on $\{H_o = \infty \}$. By \cite[Lemma~1]{ribeiro2023}, $P_Q(H_o < \infty) \le P_\kappa(H_o < \infty) < 1$ for all $Q \in \mathcal{Q}_\kappa$, so $\inf_{Q \in \mathcal{Q}_\kappa}P_Q(H_o = \infty) \ge P_\kappa(H_o = \infty) > 0$. Therefore the conditional tail satisfies $P_Q(\tau_1 \ge n \mid H_o = \infty) \le P_\kappa(H_o = \infty)^{-1}e^{-cn}$, and Theorem~\ref{t:tautail} gives exponential bounds for both the unconditional and conditional distributions. Thus the first and third summands in the right-hand side of \eqref{eq:ldproof2} are bounded from above by $\exp\{-c(\varepsilon)t\}$. Therefore, we now focus on bounding the second summand.
    
By Lemma \ref{l:N_ld_bound}, we have
\begin{equation}
    P_Q \left(
            \frac{N(t)}{t} - \frac{1}{\mutau}
            >
        \frac{\varepsilon}{6\mutau^2}
    \right)
        \leq
            \exp\{ -c(\varepsilon) t\}
\end{equation}
Therefore, again by the union bound, the result will follow if we provide a suitable upper bound for
\begin{equation}
    P_Q \left(
        \sum_{k = 1}^{\lceil(t/\mutau)(1+\varepsilon/6\mutau)\rceil}
           \left(
                \D{\tau_{k+1}} 
                -
                \D{\tau_{k}}
           \right)
            >
        \left( 
            \frac{\varepsilon}{3} + v(Q) 
        \right) t
    \right).
\end{equation}  
By \cite[Theorem 3]{ribeiro2023}, we have $v(Q) = \mudtau / \mutau$ and the above becomes
\begin{equation}
    P_Q \left(
        \sum_{k = 1}^{\lceil(t/\mutau)(1+\varepsilon/6\mutau)\rceil}
           \left(
                \D{\tau_{k+1}} 
                -
                \D{\tau_{k}}
           \right)
            >
        \frac{t}{\mutau} 
            \left( 
                \frac{\varepsilon\mutau}{3} + \mudtau 
            \right)
    \right).
\end{equation}  
Writing $m_t = \lceil(t/\mutau)(1+\varepsilon/6\mutau)\rceil$, we note that $\mudtau \le \mutau$ (since $\D{\tau_1} \le \tau_1 + 1$ and $\mutau \ge 1$) and that $\mutau$ is bounded above uniformly over $Q \in \mathcal{Q}_\kappa$ (by the uniform exponential tail of $\tau_1$). Therefore, for sufficiently small $\varepsilon$, the above probability is bounded from above by
\begin{equation}
    P_Q \left(
        \sum_{k = 1}^{m_t}
           \left(
                \D{\tau_{k+1}} 
                -
                \D{\tau_{k}}
           \right)
            >
        m_t 
            \left( 
                \frac{\varepsilon\mutau}{10} + \mudtau 
            \right)
    \right).
\end{equation}  
But now  $\D{\tau_{k+1}}-\D{\tau_{k}}$ are i.i.d. random variables with mean $\mudtau$ and a moment generating function, since they are stochastically dominated by $\tau_{k+1} - \tau_k$. Since $m_t$ is linear in $t$, a Chernoff bound finishes the proof of the result.
\end{proof}

\section{Proof of Theorem \ref{t:v}: $v(p)$ is analytic}\label{s:analyticity}
For each $p \in (0,1]$, we know that the walker has a well-defined speed $v(p)$, which is given by the limit of $\D{n}/n$ as $n$ goes to infinity. In this section, we will show that the function $p \mapsto v(p)$ is analytic on $(0,1]$. The main tool for showing this result is complex analysis.

We will follow the technique given by \cite{grego} in the context of bond percolation in~$\mathbb{Z}^d$. In a nutshell, the argument consists of seeing probabilities as functions of the parameter $p$, and then extending them to the complex numbers where the Weierstrass Theorem (Theorem \ref{t:Wei}) and his $M$-test (Theorem \ref{t:Mtest}) can be applied. The goal is then to properly bound the modulus of the complex extensions so the $M$-test can be used.

With the above discussion in mind, we will make the following abuse of notation: given an event $A$, we denote by $P_z(A)$ the analytic complex extension of the real function $p \mapsto P_p(A)$. Although we know that $P_z(A)$ is meaningless in terms of probability, this notation has the advantage of being more visually appealing and making some expressions nicer. Also, it saves us the burden of introducing a new notation for every new extension that appears in our arguments. For instance, if $P_p(A_n) = \sum q_n (1-p)^{n-i}p^i$, then $P_z(A_n) = \sum q_n(1-z)^{n-i}z^i$ for $z \in \mathbb{C}$. We will denote the complex ball of radius $r$ around $z$ as $B_{\mathbb{C}}(z,r)$.

For the proof of Theorem \ref{t:v}, we will need the exponential tail bounds on $\tau_1$ given by Theorem \ref{t:tautail} and the following technical lemmas
\begin{lemma}\label{l:An} Fix $p \in (0,1], n \in \mathbb{N}$ and $0<r<p$. Then, for any $z \in B_{\mathbb{C}}(p,r)$ the following bound holds
$$
|P_{z}(A_n)| \le \exp\{(2r/(p-r))n\}P_{p-r}(A_n)
$$
for all $A_n  \in \mathcal{F}_n$. 
\end{lemma}

\begin{lemma}\label{l:Ho} 
There exists an analytic extension of the function $p \mapsto P_p(H_o = \infty)$ to an open neighborhood of $\mathbb{C}$ containing the interval $(0, 1)$. We denote such function by $P_z(H_o = \infty)$.
\end{lemma}

\begin{lemma}\label{l:tau_analytics} Let $B_n \in \mathcal{F}_n$, then there exists $A_n \in \mathcal{F}_n$, such that
\begin{equation}\label{eq:fatorataun}
P_p(B_n, \tau_1 = n, H_o = \infty) = P_p(A_n)P_p(H_o = \infty)
\end{equation}
and that
\begin{equation}\label{eq:PAn_taun}
P_p(A_n) \le P_p(\tau_1 \ge n).
\end{equation}
\end{lemma}
We will postpone the proof of the above lemmas to the end of this section. For now, we will focus on how Theorem \ref{t:v} follows from them.

\begin{proof}[Proof of Theorem \ref{t:v}: $v(p)$ is analytic] Leveraging the renewal structure, in Theorem 3 of \cite{ribeiro2023}, the authors showed the following expression for $v(p)$
\begin{equation}\label{eq:formulavp}
    v(p) = \frac{E_p[\D{\tau_1} \; ; \; H_o = \infty ]}{E_p[{\tau_1} \; ; \; H_o = \infty ]}.
\end{equation}
We will show that both numerator and denominator are analytic functions.

\noindent \underline{$ E_p[{\tau_1} \; ; \; H_o = \infty ]$ is analytic.}  We will use Theorems \ref{t:Wei} and \ref{t:Mtest} to show that the function defined below is analytic
\begin{equation}\label{eq:Eztau}
    \begin{split}
    E_z[{\tau_1} \; ; \; H_o = \infty ] & := \sum_{n=1}^\infty nP_z(\tau_1 = n,H_o = \infty).
    \end{split}
\end{equation}
By Lemma \ref{l:tau_analytics}, there exists an $\mathcal{F}_n$-measurable event $A_n$ such that
\[
    P_p(\tau_1 = n, H_o = \infty) := P_p(A_n)P_p(H_o = \infty).
\]
We can then extend $P_p(\tau_1 = n, H_o = \infty)$ to a complex neighborhood of $(0,1]$ as the product of two functions
\begin{equation}\label{eq:pztau1n}
    P_z(\tau_1 = n, H_o = \infty) := P_z(A_n)P_z(H_o = \infty).
\end{equation}
Note that the product above is an analytic function by Lemma \ref{l:Ho} and the fact that $P_p(A_n)$ is a polynomial function of $p$ (see the proof of Lemma~\ref{l:An} below). Now, fix $p \in (0,1]$ and $r>0$ with $0 < r < p$. Then, by Lemma \ref{l:An}, it follows that for all $z \in B_{\mathbb{C}}(p,r)$
$$
    |P_z(A_n)| \le \exp\Lbrace\frac{2rn}{p-r}\Rbrace P_{p-r}(A_n),
$$
which combined with \eqref{eq:PAn_taun} of Lemma \ref{l:tau_analytics} and Theorem \ref{t:tautail} yields
$$
    |P_z(A_n)| \le \exp\Lbrace\frac{2rn}{p-r}\Rbrace P_{p-r}(\tau_1 \ge n) \le \exp\Lbrace\frac{2rn}{p-r}\Rbrace e^{-cn},
$$
for a constant $c$ which can be chosen uniformly across all $p \in (p_0,1]$, for $p_0 > 0$. This implies that by making $r$ smaller if needed, we obtain the following upper bound for all $z \in B_{\mathbb{C}}(p,r)$
\begin{equation}
    |P_z(A_n)| \le e^{-cn/2}. 
\end{equation}
Also, by choosing a smaller $r>0$ if needed, we guarantee that $z \mapsto P_z(H_o = \infty)$ is analytic in the compact ball $\bar{B}_{\mathbb{C}}(p,r)$. Therefore, there exists $C > 0$ such that
\begin{equation}
    \left|P_z(\tau_1 = n, H_o = \infty)\right| \leq C e^{-cn/2}.
\end{equation}
By Weierstrass Theorems \ref{t:Mtest} and then \ref{t:Wei}, we have that $E_z[\tau_1 ; H_o = \infty]$ is indeed well defined and analytic on $B_{\mathbb{C}}(p,r)$ (the $M$-test gives uniform convergence on the compact ball $\bar{B}_{\mathbb{C}}(p,r)$, and the Weierstrass theorem then yields analyticity on the open ball). Since $p$ was arbitrary and it extends $E_p[{\tau_1} \; ; \; H_o = \infty ]$, we obtain the desired result. \\

\noindent \underline{$ E_p[\D{\tau_1} \; ; \; H_o = \infty ]$ is analytic.} The strategy for the numerator of $v(p)$ is the same as for the denominator, but the bookkeeping differs. We write
$$
E_p[\D{\tau_1} \; ; \; H_o = \infty ] = \sum_{k=1}^\infty k\, P_p(\D{\tau_1} = k,\, H_o = \infty).
$$
Expanding $P_p(\D{\tau_1} = k,\, H_o = \infty) = \sum_{j=k}^\infty P_p(\D{j} = k, \tau_1 = j, H_o = \infty)$ (note that $\D{j} = k$ requires $j \ge k$), we obtain
$$
E_p[\D{\tau_1} \; ; \; H_o = \infty ] = \sum_{k=1}^\infty k\sum_{j=k}^\infty P_p(\D{j} = k, \tau_1 = j, H_o = \infty).
$$
Since $\{\D{j} = k\} \in \mathcal{F}_j$, by Lemma \ref{l:tau_analytics} (applied with $B_j = \{\D{j} = k\}$), there exists an event $A_{j,k} \in \mathcal{F}_j$ such that 
\begin{equation*}
P_p(\D{j} = k, \tau_1 = j, H_o = \infty) = P_p(A_{j,k})P_p(H_o = \infty).    
\end{equation*}
Then, exactly as in \eqref{eq:pztau1n}, we can extend $P_p(\D{j} = k, \tau_1 = j, H_o = \infty) $ as follows
\begin{equation}\label{eq:pzdjk}
    P_z(\D{j} = k, \tau_1 = j, H_o = \infty) = P_z(A_{j,k})P_z(H_o = \infty),
\end{equation}
where $P_z(A_{j,k})$ is a polynomial function on $z$, as before. Proceeding as in the denominator case, by Lemma~\ref{l:An} and \eqref{eq:PAn_taun}, we deduce that for $r$ sufficiently small and all $z \in B_{\mathbb{C}}(p,r)$
\begin{equation}
    |P_z(\D{j} = k, \tau_1 = j, H_o = \infty)| \le C e^{-c'j},
\end{equation}
for positive constants $C, c'$. To apply the Weierstrass $M$-test we need
$$
\sum_{k=1}^\infty k \sum_{j=k}^\infty C e^{-c'j} = C\sum_{k=1}^\infty k \cdot \frac{e^{-c'k}}{1 - e^{-c'}} < \infty,
$$
which holds since $\sum_{k=1}^\infty k e^{-c'k} < \infty$. This is enough to conclude the proof by applying Theorem \ref{t:Mtest} and then \ref{t:Wei}.

\end{proof}
\subsection{Proof of Lemma \ref{l:An}}To conclude the proof the Theorem \ref{t:v}, we show Lemmas \ref{l:An}, \ref{l:Ho} and \ref{l:tau_analytics}. But before, let us introduce some additional notation that will be needed throughout the coming proofs. We denote a realization of the first $n$ steps of the TBRW as a sequence of tree-walker tuples
\begin{equation}
((T_1,x_1), \dots, (T_n,x_n)) \equiv (\vec{T}, \vec{x})_n
\end{equation}
For each $n, i \in \mathbb{N}$ and $A_n \in \mathcal{F}_n$, let $S_{i}(A_n)$ denote the realizations $(\vec{T}, \vec{x})_n \in A_n$ such that $|T_n| = i$. 
We then let $P_{(\vec{T}, \vec{x})_n}(p)$ the probability of observing exactly $(\vec{T}, \vec{x})_n$ with the coin parameter $p$. Now we are ready for the first proof.
\begin{proof}[Proof Lemma \ref{l:An}]
We begin observing that for each $n$, we have
\begin{equation}\label{eq:hnsplit}
    P_p(A_n) = \sum_{i=0}^n\sum_{(\vec{T}, \vec{x})_n  \in S_i(A_n)}P_{(\vec{T}, \vec{x})_n}(p ).
\end{equation}
    Notice that for any pair $(\vec{T}, \vec{x})_n \in S_{i}(A_n)$, it follows that 
    \begin{equation}
       P_{(\vec{T}, \vec{x})_n}(p ) = q_n(1-p)^{n-i}p^i,
    \end{equation}
where $q_n = q_n((\vec{T}, \vec{x})_n)$ is the product of jump probabilities of the $n$ steps of the walker. Whereas $(1-p)^{n-i}p^i$ accounts for the probability of adding exactly $i$ new vertices in $n$ steps.

Now, fixing $p$ and choosing $r$ such that $0<r<p$, we have that for $z \in B_{\mathbb{C}}(p,r)$,
\begin{equation}
    |z| \le p+r \text{ and }|1-z| \le 1-p+r.
\end{equation}
Indeed, the first inequality follows from the triangular inequality, and the second, from the fact that the the distance to $1$ in $B_{\mathbb{C}}(p,r)$ is maximized at $z = p - r$. Thus, using the above inequalities, for any sequence $(\vec{T}, \vec{x})_n \in S_{i}(A_n)$, it follows that
\begin{equation}
    \begin{split}
        \left|P_{(\vec{T}, \vec{x})_n}(z) \right|
        =
            q_n|1-z|^{n-i}|z|^i
        \leq 
            q_n(1-p+r)^{n-i}(p+r)^i.
    \end{split}
\end{equation}
Dividing and multiplying by $(p-r)^{i}$, the expression for $P_{(\vec{T}, \vec{x})_n}(p-r)$ appears, yielding
\begin{equation}
    |P_{(\vec{T}, \vec{x})_n}(z)| \le \frac{(p+r)^{i}}{(p-r)^{i}}P_{(\vec{T}, \vec{x})_n}(p-r).
\end{equation}
Combining the above inequality to \eqref{eq:hnsplit} together with triangle inequality and that $(p+r)/(p-r) > 1$, we obtain
\begin{equation}
    \begin{split}
        |P_z(A_n)| & \le \sum_{i=0}^n\sum_{(\vec{T}, \vec{x})_n\in S_{i}(A_n)}\frac{(p+r)^{i}}{(p-r)^{i}}P_{(\vec{T}, \vec{x})_n}(p-r)\\
        &\le \frac{(p+r)^{n}}{(p-r)^{n}}\sum_{i=0}^n\sum_{(\vec{T}, \vec{x})_n\in S_{i}(A_n)}P_{(\vec{T}, \vec{x})_n}(p-r) \\
        & \leq 
        \left( 1 + \frac{2r}{p-r}\right)^n P_{p-r}(A_n) \\
        & \leq 
        \exp\left\{
        \frac{2rn}{p-r}\right\}
        P_{p-r}(A_n),
    \end{split}
\end{equation}
concluding the proof.
    
\end{proof}

\subsection{Proof of Lemma \ref{l:Ho}: $P_z(H_o = \infty)$ is analytic}

\begin{proof}[Proof of Lemma \ref{l:Ho}]Since $P_p(H_o = \infty) = 1 - P_p(H_o < \infty)$, it is enough to show that $P_p(H_o < \infty)$ is analytic. We know that for any $p$
\begin{equation}\label{eq:Hinfty}
    P_{p}(H_o < \infty) = \sum_{n=0}^\infty P_{p}(H_o = n).
\end{equation}
Notice that $P_{p}(H_o = n)$ is a polynomial function on $p$. To see that, just split the event $\{H_o = n\}$ into all possible trees on at most $n$ vertices and possible movements the walker can make leading it to $o$ for the first time at time $n$. Moreover, by Lemma \ref{l:An} applied to $\{H_o = n \} \in \mathcal{F}_n$, it follows that
\begin{equation}
    |P_z(H_o = n)| \le e^{2nr/(p-r)}P_{p-r}(H_o = n).
\end{equation}
On the other hand, by Theorem \ref{t:tautail}, we have that there is an absolute constant $c'$ such that
\begin{equation}\label{ineq:Hn}
    P_{p-r}(H_o = n) \le P_{p-r}(\tau_1 >  n)  \le e^{-c'n},
\end{equation}
for all $p\in (p_0,1]$. To see the inclusion of events, notice that if $H_o = n$, due to the tree structure, for any leaf the walker might have possible visited before time $n$, it must return to its father in order to visit the root at time $n$. Thus, none of the leaves visited before time $n$ can be a regeneration point. Moreover, at time $n$ the walker is at the root, so it cannot regenerate from there either. This makes $\tau_1  > n$. 

Finally, for each $n$, let $M_n$ be 
$$
M_n := e^{2rn/(p-r)}P_{p-r}(H_o = n).
$$
Then, for $r$ sufficiently small, we have that
\begin{equation}
    \sum_n M_n < \infty.
\end{equation}
Recall that the radius $r$ can be made as small as we want and the constant $c'$ does not depend on $p$ as long as $p \in (p_0,1]$, for some fixed $p_0$. Finally, by the Weierstrass $M$-test (Theorem \ref{t:Mtest}) and \eqref{eq:Hinfty} it follows that $P_z(H_o < \infty)$ is analytic, concluding the proof.

\end{proof}

\subsection{Proof of Lemma \ref{l:tau_analytics}} Finally, we move towards the proof of the last lemma needed for the proof of Theorem \ref{t:v}.
\begin{proof}[Proof of Lemma \ref{l:tau_analytics}]
    We start observing that for each $n$, 
    \begin{equation}\label{eq:taueqn}
        \begin{split}
            \Lbrace B_n, \tau_1 = n, H_o = \infty \Rbrace & = \Lbrace B_n, \zeta_K = n, H_o = \infty \Rbrace \\
            & = \bigcup_{j=1}^n \Lbrace B_n, \zeta_j = n, K = j, H_o = \infty \Rbrace \\ 
            & = \bigcup_{j=1}^n \Lbrace B_n, \zeta_j = n,\, \hitfetak_1 < \infty,\, \dots, \hitfetak_{j-1} <\infty, \hitfetak_j = \infty, H_o = \infty \Rbrace
        \end{split}
    \end{equation}
Notice that
\begin{equation}\label{eq:eventidentity}
    \begin{split}
        \Lbrace B_n,\zeta_j = n,\, \hitfetak_1 < \infty,\, \dots, \hitfetak_{j-1} <\infty, \hitfetak_j = \infty, H_o = \infty \Rbrace & \\
         = \Lbrace B_n, \zeta_j = n,\, \hitfetak_1 < n,\, \dots, \hitfetak_{j-1} < n, \hitfetak_j = \infty, H_o >n\Rbrace
    \end{split}
\end{equation}
Indeed, at time $n$ the walker is at a leaf and does not return to its father anymore. Thus, due to the tree structure of the graph, in order to visit the father of $X_{\zeta_i}$, for $i<j$, the walker must visit it before time $n$, otherwise, it would need to visit the father of $X_{\zeta_j}$ in finite time. Also, for the walker to never visit the root $o$ in the above event, it is equivalent that it does not visit it before time $n$, given that it will not return the father of $X_n$ after time $n$.
    
Then, by \eqref{eq:eventidentity} and the fact that $\zeta_j$, $\hitfetak_i$'s and $H_o$ are all stopping times, we have that
\begin{equation}\label{eq:decomp}
    \Lbrace B_n, \zeta_j = n,\, \hitfetak_1 < n,\, \dots, \hitfetak_{j-1} < n, \hitfetak_j = \infty, H_o >n\Rbrace = A_{n,j}\cap \Lbrace \hitfetak_j = \infty \Rbrace,
\end{equation}
where $A_{n,j} \in \mathcal{F}_n$ is the event below
\begin{equation}\label{eq:Anj}
A_{n,j} = \Lbrace B_n,\zeta_j = n,\, \hitfetak_1 < n,\, \dots, \hitfetak_{j-1} < n, H_o >n\Rbrace.
\end{equation}
Having the definition of $\hitfetak_j$ in mind, the identity below holds $P_p$-a.s.
\begin{equation}
      \mathds{1}\{H_{f(X_0)} = \infty \} \circ \theta_n \cdot \mathds{1}_{A_{n,j}} = \mathds{1}\{\hitfetak_j = \infty \} \cdot \mathds{1}_{A_{n,j}}.
\end{equation}
Thus, Simple Markov Property yields
\begin{equation}\label{eq:smp}
    \begin{split}
        P_p\left( A_{n,j}, \hitfetak_j = \infty \; \middle | \; \mathcal{F}_n\right) = \mathds{1}_{A_{n,j}} P_{T_n,X_n;p}\left( H_{f(X_0)} = \infty \right).
    \end{split}
\end{equation}
Recall that on $A_{n,j}$, $X_n$ is on a leaf at the bottom of $T_n$. Thus, we can couple the TBRW starting from $(T_n,X_n)$ with $(T_n,X_n)$ in $A_{n,j}$ with a TBRW starting from the nonroot tip of an edge until the first time the former visits the father of $X_n$ and the latter visits the root. In this way, the following holds $P_p$-almost surely
\begin{equation}
    \mathds{1}_{A_{n,j}} P_{T_n,X_n;p}\left( H_{f(X_0)} = \infty \right) = \mathds{1}_{A_{n,j}}P_p(H_o = \infty).
\end{equation}
Having \eqref{eq:decomp} in mind and combining the above identity with \eqref{eq:smp} gives us that
\begin{equation}
    P_p\left( B_n,\zeta_j = n,\, \hitfetak_1 < n,\, \dots, \hitfetak_{j-1} < n, \hitfetak_j = \infty, H_o >n\right) = P_p(A_{n,j})P_p(H_o = \infty).
\end{equation}
Now going back to \eqref{eq:taueqn} and using the above identity yields
\begin{equation}
    P_p(B_n,\tau_1 = n, H_o = \infty) = P_p(H_o = \infty)\sum_{j=1}^n P_p(A_{n,j}).
\end{equation}
Notice that by the definition of $A_{n,j}$ it follows that $\{A_{n,j}\}_j$ is a sequence of disjoint $\mathcal{F}_n$-measurable events. Thus, 
\begin{equation}
    P_p(B_n,\tau_1 = n, H_o = \infty) = P_p \left(\cup_{j=1}^nA_{n,j} \right)P_p(H_o = \infty),
\end{equation}
which proves the first part of the lemma by setting $A_n := \cup_{j=1}^nA_{n,j}$. To conclude the proof, observe that for a fixed $j$ by the definition of $\tau_1$, on the event $A_{n,j}$, $\tau_1 \ge \zeta_j = n$.

\end{proof}

\section{Open problems}\label{s:open}

We conclude with some open questions arising from this work.
\begin{figure}[ht]
  \centering
  \includegraphics[width=0.6\textwidth]{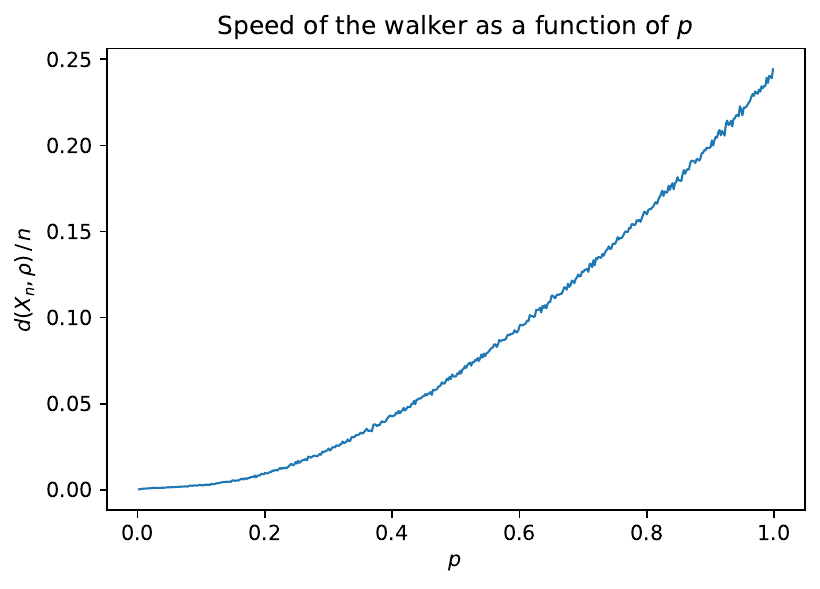}
  \caption{Estimated speed $\D{X_n}/n$ of the TBRW walker as a function
  of the growth parameter~$p$, averaged over $100$ independent runs of $2000$
  steps each.  The speed seems to be increasing in~$p$.}
  \label{fig:tbrw-speed}
\end{figure}
\begin{enumerate}[label=(\roman*)]
    \item {\bf Monotonicity of the speed.} Theorem \ref{t:v} establishes that the speed $v(p)$ is an analytic function of $p \in (0,1]$. A natural follow-up question, proposed by Y.~Peres to the second author in personal communication, is whether $v(p)$ is monotone increasing in $p$. Intuitively, as $p$ increases, the walker adds leaves more frequently, producing a stronger drift away from the root, which suggests that $v(p)$ should be increasing, see Figure~\ref{fig:tbrw-speed}. However, we do not have a formal proof of this result;
    \item {\bf Smoothness of $v(p)$ at $p=0$. } One can try to study the right derivatives of $v(p)$ at $0$, finding the correct smoothness class of the function. Figure~\ref{fig:tbrw-speed} suggests a value of $0$ for the right derivative at $0$.
    \item {\bf Phase transition on elliptic case.} For the case in which $\mathcal{L}_n = {\rm Ber}(n^{-\gamma})$, in \cite{englander2025}, the authors showed that when $\gamma > 1/2$, the walker is recurrent. The regime $\gamma \in (0,1/2]$ is believed to be transient but it is still an open problem.
    \item {\bf Large deviation principle for the empirical speed. } In Theorem \ref{t:ld}, we provide an upper bound for the large deviations of the empirical speed, it remains an open problem to obtain a matching lower bound, and characterizing the constant in the exponential as a function of $\varepsilon$. 
\end{enumerate}

\appendix 

\section{Complex Analysis results}

\begin{theorem}[Weierstrass $M$-test]\label{t:Mtest}Let $f_n$ be a sequence of complex-valued functions defined on a subset $\Omega$ of the plane and assume that there exist positive numbers $M_n$ with $|f_n(z)| \le  M_n$ for every $z\in \Omega$, and $\sum_n M_n < \infty$. Then, $\sum_n f_n$ converges uniformly on $\Omega$.
    
\end{theorem}
\begin{theorem}[Weierstrass Theorem]\label{t:Wei} Let $f_n$ be a sequence of analytic functions defined on an open subset $\Omega$ of the plane, which converges uniformly on the compact subsets of $\Omega$ to a function $f$. Then f is analytic on $\Omega$. Moreover, $f'_n$ converges uniformly on the compact subsets of $\Omega$ to $f'$.
    
\end{theorem}

\section*{Acknowledgements}

Caio Alves was partially supported by the CNPq grant 447397/2024-9.

\bibliography{ref}

@article{figueiredo2017building,
author = {Daniel Figueiredo and Giulio Iacobelli and Roberto Oliveira and Bruce Reed and Rodrigo Ribeiro},
title = {{On a random walk that grows its own tree}},
volume = {26},
journal = {Electronic Journal of Probability},
publisher = {Institute of Mathematical Statistics and Bernoulli Society},
pages = {1 -- 40},
year = {2021},
doi = {10.1214/20-EJP574},
URL = {https://doi.org/10.1214/20-EJP574}
}

@article{IRVZ22,
	author = {Giulio Iacobelli and Rodrigo Ribeiro and Glauco Valle and Leonel Zuazn{\'a}bar},
	doi = {10.3150/21-BEJ1337},
	journal = {Bernoulli},
	number = {1},
	pages = {150 -- 180},
	publisher = {Bernoulli Society for Mathematical Statistics and Probability},
	title = {{Tree builder random walk: Recurrence, transience and ballisticity}},
	url = {https://doi.org/10.3150/21-BEJ1337},
	volume = {28},
	year = {2022}}

@article{ribeiro2023,
      title = {Renewal structure of the tree builder random walk},
journal = {Stochastic Processes and their Applications},
volume = {190},
pages = {104725},
year = {2025},
issn = {0304-4149},
doi = {https://doi.org/10.1016/j.spa.2025.104725},
url = {https://www.sciencedirect.com/science/article/pii/S0304414925001668},
author = {Rodrigo Ribeiro},

}

@book{grego,
title = "Analyticity Results in Bernoulli Percolation",
author = "Agelos Georgakopoulos and Christoforos Panagiotis",
year = "2023",
month = aug,
day = "3",
doi = "10.1090/MEMO/1431",
language = "English",
isbn = "9781470467050",
series = "Memoirs of the American Mathematical Society",
publisher = "American Mathematical Society",
number = "1431",
address = "USA United States",
}

@article{solomon1975,
  author = {Solomon, Fred},
  title = {Random walks in a random environment},
  journal = {The Annals of Probability},
  volume = {3},
  number = {1},
  pages = {1--31},
  year = {1975},
  doi = {10.1214/aop/1176996444}
}

@incollection{zeitouni2004,
  author = {Zeitouni, Ofer},
  title = {Random walks in random environment},
  booktitle = {Lectures on Probability Theory and Statistics},
  series = {Lecture Notes in Mathematics},
  volume = {1837},
  pages = {189--312},
  publisher = {Springer},
  year = {2004},
  doi = {10.1007/978-3-540-39874-5_2}
}

@article{pemantle2007,
  author = {Pemantle, Robin},
  title = {A survey of random processes with reinforcement},
  journal = {Probability Surveys},
  volume = {4},
  pages = {1--79},
  year = {2007},
  doi = {10.1214/07-PS094}
}

@article{biskup2011,
  author = {Biskup, Marek},
  title = {Recent progress on the random conductance model},
  journal = {Probability Surveys},
  volume = {8},
  pages = {294--373},
  year = {2011},
  doi = {10.1214/11-PS190}
}

@article{blanc2026,
author = {Arthur Blanc-Renaudie and Camille Cazaux and Guillaume Conchon-Kerjan and Tanguy Lions and Arvind Singh},
title = {{A phase transition for the biased tree-builder random walk}},
volume = {31},
journal = {Electronic Journal of Probability},
number = {none},
publisher = {Institute of Mathematical Statistics and Bernoulli Society},
pages = {1 -- 24},
year = {2026},
doi = {10.1214/25-EJP1460},
URL = {https://doi.org/10.1214/25-EJP1460}
}

@article{sabot2015,
  author = {Sabot, Christophe and Tarr\`{e}s, Pierre},
  title = {Edge-reinforced random walk, vertex-reinforced jump process and the supersymmetric hyperbolic sigma model},
  journal = {Journal of the European Mathematical Society},
  volume = {17},
  number = {9},
  pages = {2353--2378},
  year = {2015},
  doi = {10.4171/JEMS/559}
}

@article{collevecchio2020,
  author = {Collevecchio, Andrea and Kious, Daniel and Sidoravicius, Vladas},
  title = {The branching-ruin number and the critical parameter of once-reinforced random walk on trees},
  journal = {Communications on Pure and Applied Mathematics},
  volume = {73},
  number = {1},
  pages = {210--236},
  year = {2020},
  doi = {10.1002/cpa.21860}
}

@article{avena2009,
  title={Law of Large Numbers for a Class of Random Walks in Dynamic Random Environments},
  author={Luca Avena and Frank den Hollander and Frank Redig},
  journal={Electronic Journal of Probability},
  year={2009},
  volume={16},
  pages={587-617},
  url={https://api.semanticscholar.org/CorpusID:3577592}
}

@article{sznitman_zerner1999,
  author = {Sznitman, Alain-Sol and Zerner, Martin P. W.},
  title = {A law of large numbers for random walks in random environment},
  journal = {The Annals of Probability},
  volume = {27},
  number = {4},
  pages = {1851--1869},
  year = {1999},
  doi = {10.1214/aop/1022874818}
}

@article{sznitman2000,
  author = {Sznitman, Alain-Sol},
  title = {Slowdown estimates and central limit theorem for random walks in random environment},
  journal = {Journal of the European Mathematical Society},
  volume = {2},
  number = {2},
  pages = {93--143},
  year = {2000},
  doi = {10.1007/s100970050001}
}

@article{englander2025,
author = {J{\'a}nos Engl{\"a}nder and Giulio Iacobelli and Rodrigo Ribeiro},
title = {{Recurrence, transience and degree distribution for the Tree Builder Random Walk}},
volume = {61},
journal = {Annales de l'Institut Henri Poincaré, Probabilités et Statistiques},
number = {4},
publisher = {Institut Henri Poincaré},
pages = {2553 -- 2578},
year = {2025},
doi = {10.1214/24-AIHP1513},
URL = {https://doi.org/10.1214/24-AIHP1513}
}
\bibliographystyle{plain}
\end{document}